\documentclass[12pt, reqno]{article}
\usepackage{amsfonts}
\usepackage{amsmath}
\usepackage{amssymb}
\usepackage{latexsym,bm}
\usepackage{amsthm}
\usepackage{graphicx}
\usepackage{cite}
\usepackage[colorlinks,urlcolor=blue]{hyperref}

\usepackage[top=3cm,bottom=3cm,left=3cm,right=3cm,includehead,includefoot]{geometry}
\usepackage{algorithm}
\usepackage{algorithmic}
\allowdisplaybreaks

\newtheorem{thm}{Theorem}[section]

\newtheorem{lem}{Lemma}[section]

\newtheorem{ass}{Assumption}[section]
\theoremstyle{definition}
\newtheorem{defn}{Definition}[section]
\theoremstyle{Condition}

\theoremstyle{remark}

\numberwithin{equation}{section}
\theoremstyle{example}

\numberwithin{equation}{section}
\begin{document}

\bigskip
\bigskip

\bigskip

\begin{center}

\textbf{\large An inertial primal-dual fixed point algorithm  for
composite optimization  problems}

\end{center}

\begin{center}
Meng Wen $^{1,2}$, Yu-Chao Tang$^{3}$, Jigen Peng$^{1,2}$
\end{center}

\begin{center}
1. School of Mathematics and Statistics, Xi'an Jiaotong University,
Xi'an 710049, P.R. China \\
2. Beijing Center for Mathematics and Information Interdisciplinary
Sciences, Beijing, P.R. China

 3. Department of Mathematics, NanChang University, Nanchang
330031, P.R. China
\end{center}

\footnotetext{\hspace{-6mm}$^*$ Corresponding author.\\
E-mail address: wen5495688@163.com}

\bigskip

\noindent  \textbf{Abstract} In this paper, we consider an inertial
primal-dual fixed point algorithm (IPDFP) to compute the
minimizations of the following Problem (1.1). This is a full
splitting approach, in the sense that the nonsmooth functions are
processed individually via their proximity operators. The
convergence of the IPDFP is obtained by reformulating the Problem
(1.1) to the sum of three convex functions. This work brings
together and notably extends several classical splitting schemes,
like the  primal-dual method proposed by Chambolle and Pock, and the
recent proximity algorithms of Charles A. et al designed for the
$L_{1}$/TV image denoising model. The iterative algorithm is used
for solving nondifferentiable convex optimization problems arising
in image processing. The experimental results indicate that the
proposed IPDFP iterative algorithm performs well with respect to
state-of-the-art methods.


\noindent \textbf{MR(2000) Subject Classification} 47H09, 90C25,

\section{Introduction}

The purpose of this paper is to designing and discussing an
efficient algorithmic framework with inertial version for minimizing
the following problem
$$\min_{x\in\mathcal{X}} \sum_{i=1}^{m}F_{i}(K_{i}x)+G(x),\eqno{(1.1)}$$
where $\mathcal{X}$ and $\{\mathcal{Y}_{i}\}_{i=1}^{m}$ are Hilbert
spaces, and $G\in\Gamma_{0}(\mathcal{X})$,
$F_{i}\in\Gamma_{0}(\mathcal{Y}_{i})$ respectively;
$K_{i}:\mathcal{X}\rightarrow\mathcal{Y}_{i}$ be a continuous linear
operator, for $i=1,\cdots,m$. Here and in what follows, for a real
Hilbert space $H$, $\Gamma_{0}(H)$ denotes the collection of all
proper lower semi-continuous convex functions from $H$ to
$(-\infty,+\infty]$.
\par
To that end let us first rephrase Problem (1.1) as
$$\min_{x\in\mathcal{X}^{m}} \sum_{i=1}^{m}(F_{i}(K_{i}x_{i})+\frac{1}{m}G(x_{i}))+\delta_{C}(x),\eqno{(1.2)}$$
where the Hilbert space
$\mathcal{X}^{m}:=\mathcal{X}_{1}\times\cdots\times\mathcal{X}_{m}$,
equipped with the inner product $<x,x'>:=
\sum_{i=1}^{m}<x_{i},x'_{i}>$, and
$K_{i}:\mathcal{X}_{i}\rightarrow\mathcal{Y}_{i}$ be a continuous
linear operator. The notation $x_{i}$ represents the $i$-th
component of any $x\in\mathcal{X}^{m}$, and $C$ is the space of
vectors $x\in\mathcal{X}^{m}$ such that $x_{1}=\cdots=x_{m}$. We
define the linear function
$K:\mathcal{X}^{m}\rightarrow\mathcal{Y}^{m}$ by
$Kx:=(K_{1}x_{1},\cdots,K_{m}x_{m}$), and we set
$F(Kx)=\sum_{i=1}^{m}(F_{i}(K_{i}x_{i})$ and
$\bar{G}(x)=\sum_{i=1}^{m}\frac{1}{m}G(x_{i})$  . Then the Problem
(1.2) can be rewritten as
$$\min_{x\in\mathcal{X}^{m}}F(Kx)+\bar{G}(x)+\delta_{C}(x),\eqno{(1.3)}$$

In order to solve the above problem, first we consider the following
 general optimization problem:
$$\min_{x\in\mathcal{X}}F(Kx)+\bar{G}(x)+H(x),\eqno{(1.4)}$$
where $F\in\Gamma_{0}(\mathcal{Y}^{m})$,
$\bar{G},H\in\Gamma_{0}(\mathcal{X}^{m})$.\\
 When $\bar{G}$ is differentiable on $\mathcal{X}^{m}$ and its
gradient $\nabla\bar{G}$  is $\beta$-Lipschitz continuous, for some
$\beta\in[0,+\infty[$; that is,
$$\|\nabla\bar{G}(x)-\nabla\bar{G}(x')\|\leq\beta\|x-x'\|, \forall (x, x')\in\mathcal{X}^{m}\times\mathcal{X}^{m}.$$
the primal-dual method in [1] can be used to solve (1.4).
Elaborating on the method introduced by Laurent Condat in [1] and
the method given by Nesterov in [5], we provide an iterative
algorithm for solving (1.4) which we refer to as IPDFP(inertial
primal-dual fixed point algorithm). we will give the  details of our
method in the section 3.
\par
Despite the form of (1.1) is simply, many problems in image
processing can be formulated it. For instance, the following
$L_{1}/\varphi\circ B$ model. This model minimizes the sum of the
$l_{1}$ fidelity term and the composition of a convex function with
a matrix and includes the $L_{1}$/TV denoising model and the
$L_{1}$/TV inpainting model as special cases.
$$\min_{x\in R^{2}}\|x-b\|_{1}+\delta_{C}(x)+\varphi(\circ B)(x),\eqno{(1.5)}$$
where $\varphi$ is a given convex on $R^{m}$ and $B$ is a given
$m\times n$ matrix. For the anisotropic total-variation $\varphi$ is
the norm $\|\cdot\|_{1}$ while for the isotropic total-variation .
is a linear combination of the norm  $\|\cdot\|_{2}$ in $R^{2}$. The
matrix $B$ for the both cases is the first order difference matrix.
For higher-order total-variations (see, e.g., [19-24]), $B$ may be
chosen to be a higher-order difference matrix. Problem (1.5) can be
expressed in the form of (1.1) by setting $m=2$, $G(x)=\|x-b\|_{1}$,
$F_{1}=\delta_{C}$, $K_{1}=I$ and $F_{2}=\varphi$, $K_{2}=B$. One of
the main difficulties in solving it is that $F_{i}$ and $G$ are
non-differentiable. The case often occurs in many problems we are
interested in.

\par
In this paper, the contributions of us are the following aspects:
\par
(I) We  provide  an inertial primal-dual fixed point algorithm to
solve the general Problem (1.1), which is inspired by the
primal-dual splitting method present by Laurent Condat [1] and the
method introduced by Nesterov [5]. We refer to our algorithm as
IPDFP. Firstly, when $\bar{G}$ is differentiable and $\alpha_{k}=0$
in our method, the primal-dual splitting method introduced by
 Laurent Condat [1] is a special case of our algorithm. Secondly, for $m=1$ and
 $\alpha_{k}=0$,it includes the well known first-order primal-dual algorithm proposed by Chambolle and
 Pock. Finally, when  $m=1$ and $K_{1}=I$, we can obtain the inertial forward-backward
 algorithm introduced by Dirk A. Lorenz and Thomas Pock [18].

\par
(II) Based on the idea of preconditioning techniques, we propose
simple and easy to compute diagonal preconditioners for which
convergence of the algorithm is guaranteed without the need to
compute any step size parameters and it leaves the computational
complexity of the iterations basically unchanged. As a by-product,
we show that for a certain instance of the preconditioning, the
proposed algorithm is equivalent to the old and widely unknown
primal-dual algorithm.

\par
(III) With the idea of the inertial version of the
Krasnosel'skii-Mann iterations algorithm for approximating the set
of fixed points of a nonexpansive operator and the particular inner
product defined by a symmetric positive definite map $P$ which can
be interpreted as a preconditioner or variable metric, we prove the
convergence of our method.

The rest of this paper is organized as follows. In the next section,
 we introduce some notations used throughout in the paper. In section 3, we devote
to introduce  IPDFP and SIPDFP algorithm,  and the relation between
them, we also show how the IPDFP splits  into SIPDFP and the
convergence of proposed method. In section 4, we present the
preconditioned primal-dual algorithm and give conditions under which
convergence of the algorithm is guaranteed. We propose a family of
simple and easy to compute diagonal preconditioners, which turn out
to be very efficient on many problems.  In the final section, we
show the numerical performance and efficiency of propose algorithm
through some examples in the context of large-scale
$l_{1}$-regularized logistic regression.

\section{Preliminaries }
Throughout the paper, we  denote by $\langle \cdot, \cdot\rangle$
the inner product on $\mathcal{X}$ and by $\|\cdot\|$ the norm on
$\mathcal{X}$.

\begin{ass}
The infimum of Problem (1.4) is attained. Moreover, the following
qualification condition holds
$$0\in ri(dom\, h-D\, dom\, g).$$
\end{ass}
The dual problem corresponding to the primal Problem (1.4) is
written
$$\min_{y\in\mathcal{Y}} (f+g)^{\ast}(-D^{\ast}y)+ h^{\ast}(y),$$
where $a^{\ast}$ denotes the Legendre-Fenchel transform of a
function $a$ and where $D^{\ast}$ is the adjoint of $D$. With the
Assumption 2.1, the classical Fenchel-Rockafellar duality theory
[3], [11] shows that
$$\min_{x\in\mathcal {\textbf{X}}} f(x)+g(x)+ (h\circ D)(x)=-\min_{y\in\mathcal{Y}} (f+g)^{\ast}(-D^{\ast}y)+ h^{\ast}(y).\eqno{(2.1)}$$

\begin{defn}
 Let $f$ be  a real-valued convex function on
$\mathcal{X}$, the operator prox$_{f}$ is defined by
\begin{align*}
prox_{f}&:\mathcal{X}\rightarrow\mathcal{X}\\
& x\mapsto \arg \min_{y\in
\mathcal{X}}f(y)+\frac{1}{2}\|x-y\|_{2}^{2},
\end{align*}
called the proximity operator of $f$.

\end{defn}

\begin{defn}
Let $A$ be a closed convex set of $\mathcal{X}$. Then the indicator
function of $A$ is defined as
$$
\delta_{A}(x) = \left\{
\begin{array}{l}
0,\,\,\, \,\,if x\in A,\\
\infty,\,\,\, otherwise .
\end{array}
\right.
$$
\end{defn}
It can easy see the proximity operator of the indicator function in
a closed convex subset $A$ can be reduced a projection operator onto
this closed convex set $A$. That is,
$$prox_{\iota_{A}}=proj_{A},$$
where proj is the projection operator of $A$.

\begin{defn}
(Nonexpansive operators and firmly nonexpansive operators [3]). Let
$\mathcal{H}$ be a Euclidean space (we refer to [3] for an extension
to Hilbert spaces). An operator $T : \mathcal{H} \rightarrow
{\mathcal{H}}$ is nonexpansive if and only if it satisfies
$$\|Tx-Ty\|_{2}\leq\|x-y\|_{2}\,\,\, for\,\,all\,\,\, (x,y)\in \mathcal{H}^{2}.$$
$T$ is firmly nonexpansive if and only if it satisfies one of the
following equivalent conditions:
\par
(i)$\|Tx-Ty\|_{2}^{2}\leq\langle Tx-Ty,x-y\rangle$\,\,\,
for\,\,all\,\,\, $(x,y)\in \mathcal{H}^{2}$;
\par
(ii)$\|Tx-Ty\|_{2}^{2}=\|x-y\|_{2}^{2}-\|(I-T)x-(I-T)y\|_{2}^{2}$\,\,\,
for\,\,all\,\,\, $(x,y)\in \mathcal{H}^{2}$.
\par
It is easy to show from the above definitions that a firmly
nonexpansive operator $T$ is nonexpansive.
\end{defn}
\begin{defn}
 A mapping $T : \mathcal{H} \rightarrow \mathcal{H}$ is said to be an averaged mapping, if it can
be written as the average of the identity $I$ and a nonexpansive
mapping; that is,
$$T = (1-\alpha)I +\alpha S,\eqno{(2.2)}$$
where $\alpha$ is a number in ]0, 1[ and $S : \mathcal{H}
\rightarrow \mathcal{H}$ is nonexpansive. More precisely, when (2.2)
or the following inequality (2.2) holds, we say that $T$ is
$\alpha$-averaged.
$$\|Tx-Ty\|^{2}\leq \|x-y\|^{2}-\frac{(1-\alpha)}{\alpha}\|(I -T)x-(I -T)y\|^{2},\forall x,y\in\mathcal{H}.\eqno{(2.3)}$$

\end{defn}
A 1-averaged operator is said non-expansive. A $\frac{1}{ 2}$
-averaged operator is said firmly non-expansive.

 We refer the
readers to [3] for more details. Let $M : \mathcal{H} \rightarrow
\mathcal{H}$ be a set-valued operator. We denote by $ran(M) := \{v
\in\mathcal{H} : \exists u \in\mathcal{H}, v \in Mu\}$ the range of
$M$, by $ gra(M) := \{(u, v) \in \mathcal{H}^{2} : v \in Mu\}$ its
graph, and by $M ^{-1}$ its inverse; that is, the set-valued
operator with graph ${(v, u) \in \mathcal{H}^{2} : v \in Mu}$. We
define $ zer(M) := \{u \in \mathcal{H} : 0\in Mu\}$. $M$ is said to
be monotone if $\forall(u, u' ) \in \mathcal{H}^{2},\forall(v, v' )
\in Mu\times Mu'$, $\langle u-u' , v-v' \rangle\geq 0$ and maximally
monotone if there exists no monotone operator $M'$ such that $
gra(M) \subset gra(M') \neq gra(M)$.

The resolvent $(I + M)^{-1}$ of a maximally monotone operator $M :
\mathcal{H} \rightarrow \mathcal{H}$ is defined and single-valued on
$\mathcal{H}$ and firmly nonexpansive. The subdifferential $\partial
J$ of $J\in \Gamma_{0}(\mathcal{H})$ is maximally monotone and $(I
+\partial J)^{-1} = prox_{J}$ .

 Further, let us mention some classes of operators that are used in the paper. The
operator $A$ is said to be uniformly monotone if there exists an
increasing function $\phi_{A} : [0;+1) \rightarrow [0;+1]$ that
vanishes only at 0, and
$$\langle x-y,u-v\rangle\geq\phi_{A}( \|x-y\|), \forall(x,u),(y,v)\in gra(A).\eqno{(2.5)}$$
Prominent representatives of the class of uniformly monotone
operators are the strongly monotone operators. Let
 $\gamma> 0$ be arbitrary. We say that $A$ is
$\gamma$-strongly monotone, if $\langle x-y,u-v\rangle\geq\gamma
\|x-y\|^{2}$, for all $(x,u),(y,v)\in gra(A)$.

\begin{lem}
( see[2,9-10]). Let $(\varphi^{k})_{k\in\mathbb{N}}$;
$(\delta_{k})_{k\in\mathbb{N}}$ and $(\alpha_{k})_{k\in\mathbb{N}}$
be sequences in [0;+1) such that
$\varphi^{k+1}\leq\varphi^{k}+\alpha_{k}(\varphi^{k}-\varphi^{k-1})+\delta_{k}$
for all $k\geq 1$, $\sum_{k\in\mathbb{N}}\delta_{k} < +\infty$ and
there exists a real number $\alpha$ with $0\leq \alpha_{k}
\leq\alpha < 1$ for all $k\in\mathbb{N}$. Then the following hold:\\
(i) $\sum_{k\geq1}[\varphi^{k}-\varphi^{k-1}]_{+}<+\infty$, where
$[t]_{+}=\max\{t,0\}$;\\
(ii) there exists $\varphi^{\ast}\in[0;+\infty)$ such that
$\lim_{k\rightarrow+\infty}\varphi^{k}=\varphi^{\ast}$.

\end{lem}
\begin{lem}
( [4]).Let $\tilde{M}$ be a nonempty closed and affine subset of a
Hilbert space $\mathcal {\bar{H}}$ and $T : \tilde{M} \rightarrow
\tilde{M}$ a nonexpansive operator such that $Fix(T)\neq\emptyset$.
Considering the following iterative scheme:
$$x^{k+1}=x^{k}+\alpha_{k}(x^{k}-x^{k-1})+\rho_{k}[T(x^{k}+\alpha_{k}(x^{k}-x^{k-1}))-x^{k}-\alpha_{k}(x^{k}-x^{k-1})],\eqno{(2.6)}$$
where $x^{0}$; $x^{1}$ are arbitrarily chosen in $\tilde{M}$,
 $(\alpha_{k})_{k\in \mathbb{N}}$ is nondecreasing with
$\alpha_{1} = 0$ and $0\leq\alpha_{k}\leq\alpha < 1$ for every
$n\geq1$ and $\rho, \theta, \hat{\delta}>0$ are such that
$\hat{\delta}>\frac{\alpha^{2}(1+\alpha)+\alpha\theta}{1-\alpha^{2}}$
and
$0<\rho\leq\rho_{k}<\frac{\hat{\delta}-\alpha[\alpha(1+\alpha)+\alpha\hat{\delta}+\theta]}{\hat{\delta}[1+\alpha(1+\alpha)+\alpha\hat{\delta}+\theta]}$
$\forall k\geq1$.\\
Then the following statements are true:\\
(i) $\sum_{k\in\mathbb{N}}\|x^{k+1}-x^{k}\|^{2}<+\infty$;\\
(ii) $(x^{k})_{k\in\mathbb{N}}$ converges weakly to a point in
$Fix(T)$.
\end{lem}

\section{An inertial primal-dual fixed point algorithm
 }
\subsection{Derivation of the algorithm}
In the paper [5], Nesterov proposed a modification of the heavy ball
method in order to improve the convergence rate on smooth convex
functions. The idea of Nesterov was to use the extrapolated point
$y^{k}$ for evaluating the gradient. Moreover, in order to prove
optimal convergence rates of the scheme, the extrapolation parameter
$\alpha_{k}$ must satisfy some special conditions. The scheme is
given by:
$$
\left\{
\begin{array}{l}
l^{k}=x^{k}+\alpha_{k}(x^{k}-x^{k-1}),\\
x^{k+1}=l^{k}-\bar{\lambda}_{k}\nabla f(l^{k}),
\end{array}
\right.\eqno{(3.1)}
$$
where $\bar{\lambda}_{k}=1/L$, there are several choices to define
an optimal sequence $\alpha_{k}$ [5-8].
\par
For Problem (1.4), When $\bar{G}$ is differentiable on $\mathcal
{X}$ and its gradient $\nabla\bar{G}$  is $\beta$-Lipschitz
continuous, for some $\beta\in[0,+\infty[$. Laurent Condat [1] give
the following method:
\begin{algorithmic}\label{1}
\STATE  Choose $x^{0}\in \mathcal {X}$, $ y^{0}\in \mathcal{Y}$,
relaxation parameters $(\rho_{k})_{k\in \mathbb{N}}$, and proximal
 parameters $\sigma>0$, $\tau>0$. The
iterate, for every $k\geq0$
$$
\left\{
\begin{array}{l}
\tilde{y}^{k+1}=prox_{\sigma F^{\ast}}(y^{k}+\sigma Kx^{k}),\\
\tilde{x}^{k+1}=prox_{\tau H}(x^{k}-\tau \nabla \bar{G}(x^{k})-\tau
K^{\ast}(2\tilde{y}^{k+1}-y^{k})),\\
(x^{k+1}, y^{k+1})=\rho_{k}(\tilde{x}^{k+1},
\tilde{y}^{k+1})+(1-\rho_{k})(x^{k}, y^{k}).
\end{array}
\right.\eqno{(3.2)}
$$
\end{algorithmic}

Based on the idea of Laurent Condat[1] and Nesterov[5],
 we introduce the following new algorithm for solving Problem (1.4).
\begin{algorithm}[H]
\caption{An inertial primal-dual fixed point algorithm(IPDFP).}
\begin{algorithmic}\label{1}
\STATE Initialization: Choose $x^{0}, x^{1},y^{0}, y^{1}\in \mathcal
{X}$, $
 v^{0}, v^{1}\in \mathcal{Y}$, relaxation parameters
$(\rho_{k})_{k\in \mathbb{N}}$,\\ ~~~~~~~~~~~~~~~~~~~extrapolation
parameter $\alpha_{k}$ and proximal
 parameters $\sigma>0$,$\gamma>0$, \\~~~~~~~~~~~~~~~~~~~$\tau>0$.\\
Iterations ($k\geq0$): Update $x^{k}$, $y^{k}$, $v^{k}$ as follows
$$
\left\{
\begin{array}{l}
\xi^{k}=x^{k}+\alpha_{k}(x^{k}-x^{k-1}),\\
\eta^{k}=y^{k}+\alpha_{k}(y^{k}-y^{k-1}),\\
\nu^{k}=v^{k}+\alpha_{k}(v^{k}-v^{k-1}),\\
\tilde{x}^{k+1}=prox_{\sigma H}(\xi^{k}-\sigma \eta^{k}-\sigma K^{\ast}\nu^{k}),\\
\tilde{y}^{k+1}=prox_{\gamma \bar{G}^{\ast}}(\eta^{k}+\gamma \xi^{k}),\\
\tilde{v}^{k+1}=prox_{\tau F^{\ast}}(\nu^{k}+\tau
K(2\tilde{x}^{k+1}-\eta^{k})),\\
(x^{k+1}, y^{k+1},v^{k+1})=\rho_{k}(\tilde{x}^{k+1},
\tilde{y}^{k+1},\tilde{v}^{k+1})+(1-\rho_{k})(x^{k}, y^{k},v^{k}).
\end{array}
\right.
$$
~~~~~~~~~~~~~~~~~~~End for
\end{algorithmic}
\end{algorithm}

\begin{thm}
Let $\sigma>0$, $\gamma>0$, $\tau>0$ , $(\alpha_{k})_{k\in
\mathbb{N}}$ and the sequences $(\rho_{k})_{k\in \mathbb{N}}$, be
the parameters of Algorithms 1. Let  the following conditions
hold:\\
(i)  $\sigma\gamma+\sigma\tau\|K\|^{2}<1$,\\
(ii)   $(\alpha_{k})_{k\in \mathbb{N}}$ is nondecreasing with
$\alpha_{1} = 0$ and $0\leq\alpha_{k}\leq\alpha < 1$ for every
$k\geq1$ and $\rho, \theta, \hat{\delta}>0$ are such that
$\hat{\delta}>\frac{\alpha^{2}(1+\alpha)+\alpha\theta}{1-\alpha^{2}}$
and
$0<\rho\leq\rho_{k}<\frac{\hat{\delta}-\alpha[\alpha(1+\alpha)+\alpha\hat{\delta}+\theta]}{\hat{\delta}[1+\alpha(1+\alpha)+\alpha\hat{\delta}+\theta]}$
$\forall k\geq1$.\\
Let the sequences $(x^{k},y^{k},v^{k})$ be generated by Algorithms
1. Then the sequence $\{x_{k}\}$ converges to a solution of Problem
(1.4).
\end{thm}
In the following, we would like to extend the IPDFP to solve the
 optimization problem (1.3).

\begin{algorithm}[H]
\caption{A splitting inertial primal-dual fixed point
algorithm(SIPDFP).}
\begin{algorithmic}\label{1}
\STATE Initialization: Choose $x^{0}, x^{1},y^{0}, y^{1}\in \mathcal
{X}$, $
 v^{0}_{1}, v^{1}_{1}\in \mathcal{Y}_{1},\cdots,v^{0}_{m}, v^{1}_{m}\in \mathcal{Y}_{m}$, relaxation\\ ~~~~~~~~~~~~~~~~~~~parameters
$(\rho_{k})_{k\in \mathbb{N}}$, extrapolation parameter $\alpha_{k}$
and proximal\\~~~~~~~~~~~~~~~~~~~ parameters $\sigma>0$,$\gamma>0$,
$\tau>0$.\\ Iterations:~~~~ for every  $k\geq0$
$$
\left\{
\begin{array}{l}
\xi^{k}=x^{k}+\alpha_{k}(x^{k}-x^{k-1}),\\
\tilde{x}^{k+1}=\xi^{k}-\sigma \eta^{k}-\frac{\sigma}{m} \sum_{i=1}^{m}K^{\ast}_{i}\nu^{k}_{i},\\
x^{k+1}=\rho_{k}\tilde{x}^{k+1}+(1-\rho_{k})x^{k},\\
\eta^{k}=y^{k}+\alpha_{k}(y^{k}-y^{k-1}),\\
\tilde{y}^{k+1}=prox_{\gamma G^{\ast}}(\eta^{k}+\gamma \xi^{k}),\\
y^{k+1}=\rho_{k}\tilde{y}^{k+1}+(1-\rho_{k})y^{k},\\
\nu_{i}^{k}=v^{k}_{i}+\alpha_{k}(v^{k}_{i}-v^{k-1}_{_{i}}),i=1,\cdots,m,\\
\tilde{v}^{k+1}_{i}=prox_{\tau F_{i}^{\ast}}(\nu^{k}_{i}+\tau
K_{i}(2\tilde{x}^{k+1}-\eta^{k})),i=1,\cdots,m,\\
v^{k+1}_{i}=\rho_{k}\tilde{v}^{k+1}_{i}+(1-\rho_{k})v^{k}_{_{i}},i=1,\cdots,m.
\end{array}
\right.
$$
~~~~~~~~~~~~~~~~~~~End for
\end{algorithmic}
\end{algorithm}

\begin{thm}
Let $\sigma>0$, $\gamma>0$, $\tau>0$ , $(\alpha_{k})_{k\in
\mathbb{N}}$ and the sequences $(\rho_{k})_{k\in \mathbb{N}}$, be
the parameters of Algorithms 2. Let the following conditions
hold:\\
(i)  $\sigma\gamma+\sigma\tau\sum_{i=1}^{m}\|K_{i}\|^{2}<1$,\\
(ii)   $(\alpha_{k})_{k\in \mathbb{N}}$ is nondecreasing with
$\alpha_{1} = 0$ and $0\leq\alpha_{k}\leq\alpha < 1$ for every
$k\geq1$ and $\rho, \theta, \hat{\delta}>0$ are such that
$\hat{\delta}>\frac{\alpha^{2}(1+\alpha)+\alpha\theta}{1-\alpha^{2}}$
and
$0<\rho\leq\rho_{k}<\frac{\hat{\delta}-\alpha[\alpha(1+\alpha)+\alpha\hat{\delta}+\theta]}{\hat{\delta}[1+\alpha(1+\alpha)+\alpha\hat{\delta}+\theta]}$
$\forall k\geq1$.\\
Let the sequences $(x^{k},y^{k},v^{k})$ be generated by Algorithms
2. Then the sequence $\{x_{k}\}$ converges to a solution of Problem
(1.3).
\end{thm}

\subsection{Proofs of convergence}
Proof of Theorem 3.1 for Algorithm 1. By the idea of Laurent
Condat[1], we know that Algorithm 1 has the structure of a
forward-backward iteration, when expressed in terms of nonexpansive
operators on $\mathcal {Z} :=  \mathcal {X}\times\mathcal
{X}\times\mathcal {Y}$, equipped with a particular inner product.
\par
Let the inner product $\langle\cdot,\cdot\rangle_{I}$ in $\mathcal
{Z}$ be defined as
$$\langle z,z'\rangle:=\langle x,x'\rangle+\langle y,y'\rangle+\langle v,v'\rangle,~~~~~\forall z=(x,y,v),~ z'=(x',y',v')\in\mathcal {Z}.$$
By endowing $\mathcal {Z}$ with this inner product, we obtain the
Euclidean space denoted by $\mathcal {Z}_{I}$ . Let us define the
bounded linear operator on $\mathcal {Z}$,

$$
P:= \left(
  \begin{array}{ccccccc}
    x  \\
    y  \\
    v  \\

  \end{array}
\right)\mapsto\left(
  \begin{array}{ccccccc}
    \frac{1}{\sigma} & -I  & -K^{\ast}\\
    -I & \frac{1}{\gamma}  &  0\\
    -K & 0  & \frac{1}{\tau}\\

  \end{array}
\right)\left(
  \begin{array}{ccccccc}
    x  \\
    y  \\
    v  \\

  \end{array}
\right).\eqno{(3.3)}
$$
From the condition (i), we can easily check that $P$ is positive
definite. Hence, we can define another inner product
$\langle\cdot,\cdot\rangle_{P}$ and norm
$\|\cdot\|_{P}=\langle\cdot,\cdot\rangle_{P}^{\frac{1}{2}}$ in
$\mathcal {Z}$ as
$$\langle z,z'\rangle_{P}=\langle z,Pz'\rangle_{I}.\eqno{(3.4)}$$
We denote by $\mathcal {Z}_{P}$ the corresponding Euclidean space.
\par
For every $k\in\mathbb{N}$, the following inclusion is satisfied by
$\tilde{z}^{k+1} := (\tilde{x}^{k+1},
\tilde{y}^{k+1},\tilde{v}^{k+1})$ computed by Algorithms 1:

 $$ 0\in \left(
  \begin{array}{ccccccc}
    \partial H &I & K^{\ast} \\
    -I & \partial G^{\ast} &0\\
    -K &0  & \partial F^{\ast} \\

  \end{array}
\right)\left(
  \begin{array}{ccccccc}
    \tilde{x}^{k+1} \\
   \tilde{y}^{k+1} \\
    \tilde{v}^{k+1}\\

  \end{array}
\right)+\left(
  \begin{array}{ccccccc}
   \frac{ 1}{\sigma} &-I & -K^{\ast} \\
    I & \frac{1}{\gamma} &0\\
    -K &0  & \frac{1}{\tau} \\

  \end{array}
\right)\left(
  \begin{array}{ccccccc}
    \tilde{x}^{k+1}-\xi^{k} \\
   \tilde{y}^{k+1} -\eta^{k}\\
    \tilde{v}^{k+1}-\nu^{k}\\

  \end{array}
\right)
$$
By set $$\varpi^{k}=(\xi^{k},\eta^{k},\nu^{k}),
 A:= \left(
  \begin{array}{ccccccc}
    \partial H &I & K^{\ast} \\
    -I & \partial G^{\ast} &0\\
    -K &0  & \partial F^{\ast} \\

  \end{array}
\right),
$$
 it also can be written as follows:
$$
\left\{
\begin{array}{l}
\varpi^{k}=z^{k}+\alpha_{k}(z^{k}-z^{k-1}),\\
\tilde{z}^{k+1}:=(I+P^{-1}\circ A)^{-1}(\varpi^{k}).
\end{array}
\right.\eqno{(3.5)}
$$
Considering  the relaxation step, we obtain
$$
\left\{
\begin{array}{l}
\varpi^{k}=z^{k}+\alpha_{k}(z^{k}-z^{k-1}),\\
\tilde{z}^{k+1}:=(I+P^{-1}\circ A)^{-1}(\varpi^{k}),\\
z^{k+1}:=\rho_{k}(I+P^{-1}\circ
A)^{-1}(\varpi^{k})+(1-\rho_{k})\varpi^{k}.
\end{array}
\right.\eqno{(3.6)}
$$

Set $M=P^{-1}\circ A$, then
\par
The operator $(x,y,v)\mapsto\partial H\times\partial G^{\ast}
\times\partial F^{\ast}$ is maximally monotone in $\mathcal {Z}_{I}$
by Propositions 23.16 of [3]. Moreover, the skew operator:
$$
\left(
  \begin{array}{ccccccc}
    x \\
    y\\
    v\\

  \end{array}
\right)\mapsto \left(
  \begin{array}{ccccccc}
    0 &I & K^{\ast} \\
    -I &  0 &0\\
    -K &0  &0 \\

  \end{array}
\right)\left(
  \begin{array}{ccccccc}
    x  \\
    y  \\
    v  \\

  \end{array}
\right)
$$
is maximally monotone in $\mathcal {Z}_{I}$ [3, Example 20.30] and
has full domain. Hence, $A$ is maximally monotone [3, Corollary
24.4(i)]. Thus, $M$ is monotone in $\mathcal {Z}_{P}$ and, from the
injectivity of $P$ , $M$ is maximally monotone in $\mathcal {Z}_{P}$
. Set $T=(I+M)^{-1}$, then by [3, Corollary 23.8], we know that
$T\in\mathcal {A}(\mathcal {Z}_{P},\frac{1}{2})$.
 In particular, it is non-expansive. Since $P^{-1}$ and $L$ are bounded and the norms $\|\cdot\|_{I}$ and $\|\cdot\|_{P}$
are equivalent, so from  conditions (i)-(ii)  and Lemma 2.2 we have
that the  iterative scheme defined by (3.6)
satisfies the following statements:\\
(i) $\sum_{k\in\mathbb{N}}\|z^{k+1}-z^{k}\|_{P}^{2}<+\infty$;\\
(ii) $(z^{k})_{k\in\mathbb{N}}$ converges to a point in $Fix(T)$.\\
Then the sequence $\{x^{k}\}$ converges to a solution of Problem
(1.4).

 Elaborating on Theorem 3.1,
 we are now ready to establish the
Theorem 3.2.

By the notation in Section 1, we know that, for any
$\textbf{y}=(y_{1},\cdots,y_{m})\in\mathcal
{Y}_{1}\times\cdots\times \mathcal {Y}_{m}$,
$F^{\ast}(y_{1},\cdots,y_{m})=(F_{1}^{\ast}y_{1},\cdots,F_{m}^{\ast}y_{m})$,
$K^{\ast}(y_{1},\cdots,y_{m})=(K^{\ast}_{1}y_{1},\cdots,K^{\ast}_{m}y_{m})$,
$prox_{\tau F^{\ast}}=(prox_{\tau
F^{\ast}_{1}}(y_{1}),\cdots,prox_{\tau
F^{\ast}_{m}}(y_{m}))$,$\|K\|^{2}=\|\Sigma_{i=1}^{m}K_{i}^{\ast}K_{i}\|$.
When $H(x)=\delta_{C}(x)$, where $C$ is the space of vectors
$x\in\mathcal{X}^{m}$, we know that for any $x\in\mathcal{X}^{m}$,
$proj_{C}(x)=(\bar{x},\cdots,\bar{x})$ where $\bar{x}$ is the
average of vector $x$, i.e., $\bar{x}=m^{-1}\sum_{i=1}^{m}x_{i}$.
Consequently, the components of $\tilde{x}^{k+1}$ in  Algorithm 1
are equal and coincide with $\xi^{k}-\sigma
\eta^{k}-\frac{\sigma}{m} \sum_{i=1}^{m}K^{\ast}_{i}\nu^{k}_{i}$.
Therefore, we can obtain Algorithm 2 by Algorithm 1, and we can
obtain the convergence of Theorem 3.2 directly  by Theorem 3.1.

\subsection{Connections to other algorithms}
    We will further establish the connections to other existing
methods.\\
\textbf{Primal-dual algorithms} \\
If the term $\bar{G}$ is absent of the Problem (1.4), and
$\alpha_{k}\equiv0$ , the Algorithms 1 boils down to the primal-dual
algorithms of Chambolle and Pock [12], which have been proposed in
other forms in [13, 14].\\
\textbf{Forward-backward splitting} \\
If the term $F\circ K=0$,  $\bar{G}$ is differentiable  and its
gradient $\nabla\bar{G}$  is Lipschitz continuous, $\alpha_{k}=0$ in
Algorithms 1. We obtain exactly the popular forward-backward
splitting algorithm  for minimizing the sum of a smooth and a
non-smooth convex function. See [15,16].\\

\section{Preconditioning}
\subsection{Convergence of the Preconditioned algorithm}
In the context of saddle point problems, Pock and Chambolle [17]
proposed a preconditioning of the form

$$ B:= \left(
  \begin{array}{ccccccc}
    \tilde{T}^{-1} & -K^{\ast} \\
   - K & \Sigma^{-1} \\

  \end{array}
\right)
$$

where $\tilde{T}$ and $\Sigma$ are selfadjoint, positive definite
maps. A condition for the positive definiteness of $P$ follows from
the following Lemma.
\begin{lem}
([14]). Let $A_{1}$, $A_{2}$ be symmetric positive definite maps and
$M$ a bounded operator. If
$\|A_{2}^{-\frac{1}{2}}MA_{1}^{-\frac{1}{2}} \| < 1$, then
$$ A:= \left(
  \begin{array}{ccccccc}
    A_{1} & M^{\ast} \\
    M & A_{2} \\

  \end{array}
\right)
$$

is positive definite.
\end{lem}
Based on the idea of Pock and Chambolle, we present a
preconditioning of the form
$$
\bar{P}:= \left(
  \begin{array}{ccccccc}
    \Sigma^{-1} & -I  & -K^{\ast}\\
    -I & \Upsilon^{-1}  &  0\\
    -K & 0  & \tilde{T}^{-1}\\

  \end{array}
\right),\eqno{(4.1)}
$$
where $\Sigma$, $\Upsilon$ and $\tilde{T}$ are selfadjoint, positive
definite maps. A condition for the positive definiteness of
$\bar{P}$ follows from the following Lemma.
\begin{lem}
 Let $\Sigma$, $\Upsilon$ and $\tilde{T}$ be symmetric positive definite maps and
$\bar{P}$ a bounded operator. If
$\|\Sigma^{\frac{1}{2}}\Upsilon^{\frac{1}{2}}\|^{2}
+\|\Sigma^{\frac{1}{2}}K^{\ast}\tilde{T}^{\frac{1}{2}}\|^{2}<1$,
then the matrix $\bar{P}$ defined in (4.1) is symmetric and positive
definite .
\end{lem}

\begin{proof}

Due to the structure of $\bar{P}$, we have that
$$
\langle\left(
  \begin{array}{ccccccc}
    x \\
    y\\
    v\\

  \end{array}
\right), \bar{P}\left(
  \begin{array}{ccccccc}
    x \\
    y\\
    v\\

  \end{array}
\right)\rangle=\langle x,\Sigma^{-1}x\rangle+\langle
y,\Upsilon^{-1}y\rangle+\langle v,\tilde{T}^{-1}v\rangle-2\langle
y+K^{\ast}v,x\rangle,
$$
set $$D=(I,K^{\ast}), u=(y,v)^{T},
 M:= \left(
  \begin{array}{ccccccc}
    \Upsilon & 0 \\
    0 & \tilde{T} \\

  \end{array}
\right)
$$
 then estimate the
middle term from below by Cauchy- Schwarz and Young＊s inequality
and get for every $\varepsilon > 0$ that
\begin{align*}
-2\langle y+K^{\ast}v,x\rangle&=-2\langle
Du,x\rangle=-2\langle\Sigma^{\frac{1}{2}}DM^{\frac{1}{2}}M^{-\frac{1}{2}}u,\Sigma^{-\frac{1}{2}}x\rangle\\
&\geq-2\|\Sigma^{\frac{1}{2}}DM^{\frac{1}{2}}M^{-\frac{1}{2}}u\|\|\Sigma^{-\frac{1}{2}}x\|\\
&\geq-(\varepsilon\|\Sigma^{\frac{1}{2}}DM^{\frac{1}{2}}\|^{2}\|u\|^{2}_{M^{-\frac{1}{2}}}+\frac{1}{\varepsilon}\|x\|^{2}_{\Sigma^{-\frac{1}{2}}}).
\end{align*}
Since
$\|\Sigma^{\frac{1}{2}}DM^{\frac{1}{2}}\|^{2}=\|\Sigma^{\frac{1}{2}}\Upsilon^{\frac{1}{2}}\|^{2}
+\|\Sigma^{\frac{1}{2}}K^{\ast}\tilde{T}^{\frac{1}{2}}\|^{2}<1$, so
we have
$$
\langle\left(
  \begin{array}{ccccccc}
    x \\
    y\\
    v\\

  \end{array}
\right), \bar{P}\left(
  \begin{array}{ccccccc}
    x \\
    y\\
    v\\

  \end{array}
\right)\rangle\geq(1-\varepsilon\|\Sigma^{\frac{1}{2}}DM^{\frac{1}{2}}\|^{2})\|u\|^{2}_{M^{-\frac{1}{2}}}+(1-\frac{1}{\varepsilon})\|x\|^{2}_{\Sigma^{-\frac{1}{2}}}>0.
$$

\end{proof}

 Now, we study preconditioning
techniques for the inertial primal-dual fixed point
algorithm(IPDFP), then we obtain the following algorithm.
\begin{algorithm}[H]
\caption{An inertial primal-dual fixed point algorithm with
preconditioning (IPDFP$^{2}$).}
\begin{algorithmic}\label{1}
\STATE Initialization: Choose $x^{0}, x^{1},y^{0}, y^{1}\in \mathcal
{X}$, $
 v^{0}, v^{1}\in \mathcal{Y}$, relaxation parameters
$(\rho_{k})_{k\in \mathbb{N}}$,\\ ~~~~~~~~~~~~~~~~~~~extrapolation
parameter $\alpha_{k}$ and positive definite maps $\Sigma$, $\Upsilon$ and $\tilde{T}$.\\
Iterations ($k\geq0$): Update $x^{k}$, $y^{k}$, $v^{k}$ as follows
$$
\left\{
\begin{array}{l}
\xi^{k}=x^{k}+\alpha_{k}(x^{k}-x^{k-1}),\\
\eta^{k}=y^{k}+\alpha_{k}(y^{k}-y^{k-1}),\\
\nu^{k}=v^{k}+\alpha_{k}(v^{k}-v^{k-1}),\\
\tilde{x}^{k+1}=prox_{\Sigma H}(\xi^{k}-\Sigma \eta^{k}-\Sigma K^{\ast}\nu^{k}),\\
\tilde{y}^{k+1}=prox_{\Upsilon \bar{G}^{\ast}}(\eta^{k}+\Upsilon \xi^{k}),\\
\tilde{v}^{k+1}=prox_{\tilde{T} F^{\ast}}(\nu^{k}+\tilde{T}
K(2\tilde{x}^{k+1}-\eta^{k})),\\
(x^{k+1}, y^{k+1},v^{k+1})=\rho_{k}(\tilde{x}^{k+1},
\tilde{y}^{k+1},\tilde{v}^{k+1})+(1-\rho_{k})(x^{k}, y^{k},v^{k}).
\end{array}
\right.
$$
~~~~~~~~~~~~~~~~~~~End for
\end{algorithmic}
\end{algorithm}
It turns out that the resulting method converges under appropriate
conditions.
\begin{thm}
In the setting of Theorem 3.1
let the following conditions holds :\\
(i) $\|\Sigma^{\frac{1}{2}}\Upsilon^{\frac{1}{2}}\|^{2}
+\|\Sigma^{\frac{1}{2}}K^{\ast}\tilde{T}^{\frac{1}{2}}\|^{2}<1$;\\
(ii)   $(\alpha_{k})_{k\in \mathbb{N}}$ is nondecreasing with
$\alpha_{1} = 0$ and $0\leq\alpha_{k}\leq\alpha < 1$ for every
$k\geq1$ and $\rho, \theta, \hat{\delta}>0$ are such that
$\hat{\delta}>\frac{\alpha^{2}(1+\alpha)+\alpha\theta}{1-\alpha^{2}}$
and
$0<\rho\leq\rho_{k}<\frac{\hat{\delta}-\alpha[\alpha(1+\alpha)+\alpha\hat{\delta}+\theta]}{\hat{\delta}[1+\alpha(1+\alpha)+\alpha\hat{\delta}+\theta]}$
$\forall k\geq1$.\\
Then the sequence $\{x^{k}\}$ generated by the Algorithm 3 converges
to a solution of Problem (1.4).

\end{thm}
\begin{proof}
As shown in Lemma 4.2, the condition
$\|\Sigma^{\frac{1}{2}}\Upsilon^{\frac{1}{2}}\|^{2}
+\|\Sigma^{\frac{1}{2}}K^{\ast}\tilde{T}^{\frac{1}{2}}\|^{2}<1$
ensure that the matrix $\bar{P}$ defined in (4.1) is symmetric and
positive definite. Therefore, with the same proof of Theorem 3.1, we
can obtain Theorem 4.1.

\end{proof}
For selfadjoint, positive definite maps $\Sigma$, $\Upsilon$,
$\tilde{T}$, we consider the following algorithm which we shall
refer to as a preconditioned splitting inertial primal-dual fixed
point algorithm(PSIPDFP).

\begin{algorithm}[H]
\caption{Preconditioned splitting inertial primal-dual fixed point
algorithm(PSIPDFP).}
\begin{algorithmic}\label{1}
\STATE Initialization: Choose $x^{0}, x^{1},y^{0}, y^{1}\in \mathcal
{X}$, $
 v^{0}_{1}, v^{1}_{1}\in \mathcal{Y}_{1},\cdots,v^{0}_{m}, v^{1}_{m}\in \mathcal{Y}_{m}$, relaxation\\ ~~~~~~~~~~~~~~~~~~~parameters
$(\rho_{k})_{k\in \mathbb{N}}$, extrapolation parameter $\alpha_{k}$
and positive definite \\~~~~~~~~~~~~~~~~~~~maps $\Sigma$, $\Upsilon$ and $\tilde{T}$.\\
Iterations:~~~~ for every  $k\geq0$
$$
\left\{
\begin{array}{l}
\xi^{k}=x^{k}+\alpha_{k}(x^{k}-x^{k-1}),\\
\tilde{x}^{k+1}=\xi^{k}-\Sigma \eta^{k}-\frac{1}{m} \Sigma\sum_{i=1}^{m}K^{\ast}_{i}\nu^{k}_{i},\\
x^{k+1}=\rho_{k}\tilde{x}^{k+1}+(1-\rho_{k})x^{k},\\
\eta^{k}=y^{k}+\alpha_{k}(y^{k}-y^{k-1}),\\
\tilde{y}^{k+1}=prox_{\Upsilon G^{\ast}}(\eta^{k}+\Upsilon \xi^{k}),\\
y^{k+1}=\rho_{k}\tilde{y}^{k+1}+(1-\rho_{k})y^{k},\\
\nu_{i}^{k}=v^{k}_{i}+\alpha_{k}(v^{k}_{i}-v^{k-1}_{_{i}}),i=1,\cdots,m,\\
\tilde{v}^{k+1}_{i}=prox_{\tilde{T}
F_{i}^{\ast}}(\nu^{k}_{i}+\tilde{T}
K_{i}(2\tilde{x}^{k+1}-\eta^{k})),i=1,\cdots,m,\\
v^{k+1}_{i}=\rho_{k}\tilde{v}^{k+1}_{i}+(1-\rho_{k})v^{k}_{_{i}},i=1,\cdots,m.
\end{array}
\right.
$$
~~~~~~~~~~~~~~~~~~~End for
\end{algorithmic}
\end{algorithm}
\begin{thm}
In the setting of Theorem 3.2 let the following conditions holds :\\
(i) $\|\Sigma^{\frac{1}{2}}\Upsilon^{\frac{1}{2}}\|^{2}
+\sum_{i=1}^{m}\|\Sigma^{\frac{1}{2}}K^{\ast}_{i}\tilde{T}^{\frac{1}{2}}\|^{2}<1$;\\
(ii)   $(\alpha_{k})_{k\in \mathbb{N}}$ is nondecreasing with
$\alpha_{1} = 0$ and $0\leq\alpha_{k}\leq\alpha < 1$ for every
$k\geq1$ and $\rho, \theta, \hat{\delta}>0$ are such that
$\hat{\delta}>\frac{\alpha^{2}(1+\alpha)+\alpha\theta}{1-\alpha^{2}}$
and
$0<\rho\leq\rho_{k}<\frac{\hat{\delta}-\alpha[\alpha(1+\alpha)+\alpha\hat{\delta}+\theta]}{\hat{\delta}[1+\alpha(1+\alpha)+\alpha\hat{\delta}+\theta]}$
$\forall k\geq1$.\\
Then the sequence $\{x^{k}\}$ generated by the Algorithm 4 converges
to a solution of Problem (1.3).

\end{thm}
\begin{proof}

 Set
$$\bar{K}^{\ast}:= \left(
  \begin{array}{ccccccc}
    K_{1}^{\ast} \\
   \vdots \\
    K_{m}^{\ast} \\
  \end{array}
\right), \bar{\textbf{P}}:=  \left(
  \begin{array}{ccccccc}
    \Sigma^{-1} & -I  & -\bar{K}^{\ast}\\
    -I & \Upsilon^{-1}  &  0\\
    -\bar{K} & 0  & \tilde{T}^{-1}\\

  \end{array}
\right).
$$

 Then from  the condition $\|\Sigma^{\frac{1}{2}}\Upsilon^{\frac{1}{2}}\|^{2}
+\sum_{i=1}^{m}\|\Sigma^{\frac{1}{2}}K^{\ast}_{i}\tilde{T}^{\frac{1}{2}}\|^{2}<1$
and Lemma 4.2, we can know that $\bar{\textbf{P}}$ is symmetric and
positive definite. Hence, the convergence of the Algorithm 4 to an
optimal solution  of (1.3)  follows from the weak convergence of the
Algorithm 2.

\end{proof}
\subsection{Diagonal Preconditioning}
In this section, we show how we can choose pointwise step sizes for
both the primal and the dual variables that will ensure the
convergence of the algorithm. The next result is an adaption of the
preconditioner proposed in [17].

\begin{lem}
Let $\Sigma= diag(\sigma_{1}, \cdots , \sigma_{n})$, $\Upsilon=
diag(\gamma_{1}, \cdots , \gamma_{n})$ and $\tilde{T} =
diag(\tau_{1},\cdots , \tau_{m})$, then we can know that $M=
diag(\gamma_{1}, \cdots , \gamma_{n},\tau_{1},\cdots , \tau_{m})$.
In particular, we set $M= diag(\gamma_{1}, \cdots ,
\gamma_{n},\tau_{1},\cdots , \tau_{m})=diag(\varphi_{1}, \cdots ,
\varphi_{n+m})$ with
$$\sigma_{j}=\frac{1}{\sum_{i=1}^{n+m}|D_{i,j}|^{2-s}},\varphi_{i}=\frac{1}{\sum_{j=1}^{n}|D_{i,j}|^{s}},\eqno{(4.2)}$$
then for any $s\in[0,2]$\\
$$\|\Sigma^{\frac{1}{2}}\Upsilon^{\frac{1}{2}}\|^{2}
+\|\Sigma^{\frac{1}{2}}K^{\ast}\tilde{T}^{\frac{1}{2}}\|^{2}=\|\Sigma^{\frac{1}{2}}DM^{\frac{1}{2}}\|^{2}\leq1.\eqno{(4.3)}$$

\end{lem}
\begin{proof}
In order to prove the inequality, we need to find an upper bound on
$\|\Sigma^{\frac{1}{2}}DM^{\frac{1}{2}}\|^{2}$. From the proof of
[17,Lemma 2] we can obtain the results directly.

\end{proof}

\section{Numerical experiments}
We consider the problem of $l_{1}$-regularized logistic regression.
Denoting by $m$ the number of observations and by $q$ the number of
features, the optimization problem writes
$$\inf_{x\in \mathbb{R}^{q}}\frac{1}{m}\sum_{i=1}^{m}\log(1+e^{-y_{i}a_{i}^{T}x})+\tau\|x\|_{1},\eqno{(8.1)}$$
where the $(y_{i})_{i=1}^{m}$ are in $\{-1,+1\}$, the
$(a_{i})_{i=1}^{m}$ are in $\mathbb{R}^{q}$, and $\tau>0$ is a
scalar. Let $(\mathcal{W})_{n=1}^{N}$ indicate a partition of $\{1,
. . . ,m\}$. The optimization problem then writes

$$\inf_{x\in \mathbb{R}^{q}}\sum_{n=1}^{N}\sum_{i\in\mathcal{W}_{n}}\frac{1}{m}\log(1+e^{-y_{i}a_{i}^{T}x})+\tau\|x\|_{1},\eqno{(8.2)}$$

or, splitting the problem between the batches

$$\inf_{x\in \mathbb{R}^{N^{q}}}\sum_{n=1}^{N}(\sum_{i\in\mathcal{W}_{n}}\frac{1}{m}\log(1+e^{-y_{i}a_{i}^{T}x_{n}})+\frac{\tau}{N}\|x_{n}\|_{1})+\iota_{\mathcal{C}(x)},\eqno{(8.3)}$$
where $x = (x_{1}, ..., x_{N})$ is in  $\mathbb{R}^{N^{q}}$. It is
easy to see that Problems (8.1), (8.2) and (8.3) are equivalent and
Problem (8.3) is in the form of (6.2).
\section{Conclusion}

In  this paper, we introduced a new framework for stochastic
coordinate descent and used on a algorithm called ADMMDS$^{+}$. As a
byproduct, we obtained a stochastic approximation algorithm with
dynamic stepsize which can be used to handle distinct data blocks
sequentially. We also obtained an asynchronous distributed algorithm
with dynamic stepsize which enables the processing of distinct
blocks on different machines.

\noindent \textbf{Acknowledgements}

This work was supported by the National Natural Science Foundation
of China (11131006, 41390450, 91330204, 11401293), the National
Basic Research Program of China (2013CB 329404), the Natural Science
Foundations of Jiangxi Province (CA20110\\
7114, 20114BAB 201004).

\end{document}